\documentclass[12pt,a4paper]{amsart}
\usepackage{amsfonts, amssymb}
\usepackage[all]{xy}

\def\B{\mathcal B}

\newtheorem{thm}{Theorem}

\newtheorem{prop}{Proposition}
\newcommand{\E}{\ensuremath{\mathbb E}}

\numberwithin{equation}{section} \numberwithin{thm}{section}
\numberwithin{lem}{section} \numberwithin{problem}{section}

\newcommand{\p}{\mathbf p}
\newcommand{\q}{\mathbf q}

\newcommand{\Pp}{\mathcal{P}}

\newcommand{\Z}{\mathbb{Z}}

\newcommand{\F}{\mathbb F}
\begin{document}

\title{Infinite Sidon sequences}
\author{Javier Cilleruelo}
\thanks{This work was supported by Grant MTM 2011-22851 of MICINN (Spain)}
\address{Instituto de Ciencias Matem\'aticas
 (CSIC-UAM-UC3M-UCM) and Departamento de Matem\'aticas\\
Universidad Aut\'onoma  de Madrid\\
28049 Madrid, Spain } \email{franciscojavier.cilleruelo@uam.es}


\begin{abstract}We present a method to construct dense infinite Sidon sequences based on the discrete logarithm.  We give an explicit construction of an infinite Sidon sequence $A$ with $A(x)= x^{\sqrt 2-1+o(1)}$. Ruzsa proved the existence of a Sidon sequence with similar counting function but his proof was not constructive.

Our method generalizes to  $B_h$ sequences: For all $h\ge 3$,  there is a $B_h$ sequence $A$ such that $A(x)=x^{\sqrt{(h-1)^2+1}-(h-1)+o(1)}$.
\end{abstract}

\maketitle

\section{Introduction}
 According to Erd\H os \cite{E}, in 1932 Simon Sidon asked  him about the growing of those infinite sequences $A$  with the property that all sums $a+a',\ a\le a',\ a,a'\in A$ are distinct. Later Erd\H os named them Sidon sequences. Sidon had found one with counting function $A(x)\gg x^{1/4}$ and  Erd\H os observed that the greedy algorithm, described below, provides another with $A(x)\gg x^{1/3}$.

 Starting with $a_1=1$, let us define $a_n$ to be the smallest integer greater than $a_{n-1}$ and such that the set $\{a_1,\dots ,a_n\}$ is a Sidon set. Since there are at most $(n-1)^3$ distinct elements of the form $a_i+a_j-a_k,\ 1\le i,j,k\le n-1$, it is then clear that $a_n\le (n-1)^3+1$ and the counting function of the sequence $A$ generated by this algorithm (the greedy algorithm)  certainly  satisfies $A(x)\ge x^{1/3}$.

Erd\H os conjectured that for any $\epsilon >0$ should exist a Sidon sequence  with $A(x)\gg x^{1/2-\epsilon}$, but the sequence given by the greedy algorithm was, for almost 50 years, the densest example known. That was until that, in 1981, Atjai, Komlos and Szemeredi \cite{AKS} proved the existence of an infinite Sidon sequence such that $A(x)\gg (x\log x)^{1/3}$. The main tool was a remarkable new result in graph theory that they proved in that seminal paper. They wrote:

\emph{The task of constructing a denser sequence has so far resisted all efforts, both constructive and random methods. Here we use a random construction for giving a sequence which is denser than the above trivial one.}

 So that, it was a surprise when Ruzsa \cite{Ru} overcome the barrier of the exponent $1/3$, proving the existence of an infinite Sidon sequence $A$ with $A(x)=x^{\sqrt 2-1+o(1)}$.
The starting point of Ruzsa's approach was the sequence $(\log p )$ where $p$ runs over all the prime numbers.
Ruzsa's proof is not constructive. For each $\alpha\in [1,2]$  he considered a sequence $A_{\alpha}=(a_p)_{p\in \Pp}$  where each $a_p$ is built using the binary digits of $\alpha\log p$. What Ruzsa proved is that for almost all $\alpha\in [1,2]$ the sequence $A_{\alpha}$ is nearly a Sidon sequence in the sense that removing not too many elements from the sequence
it is possible to destroy all the repeated sums that eventually appear.

Here we present a  method to construct explicitly dense infinite Sidon sequences. It has been inspired by the finite Sidon set $$A=\{\log_g p:\ p\text{ prime },\ p\le \sqrt q\},$$
where  $g$ a generator of $\F_q^*$ and  $\log_g x$ denotes the discrete logarithm of $x$ in $\F_q^*$. The set $A$
is indeed a Sidon set in $\Z_{q-1}$ with size $|A|=\pi(\sqrt q)\sim 2\sqrt q/\log q$.
Despite the simplicity of the construction of this finite Sidon set we have not seen it previously in the literature.

Our main result is Theorem \ref{sqrt2} but to warm up we construct first an infinite Sidon sequence $A=(a_p)_{p\in \Pp}$ indexed with all the prime numbers  with an easy explicit expression for the elements $a_p$. This is the first time that an infinite Sidon sequence $A$ with $A(x)\gg x^{\delta}$ for some $\delta>1/3$ is constructed explicitly.
\begin{thm}\label{sqrt5}
Let $A:=A_{\overline q,c}=(a_p)_{p\in \Pp}$ be the sequence constructed in section 2. We have that for $c=\frac{3-\sqrt 5}2$ it is an infinite Sidon sequence with $$A(x)=x^{\frac{3-\sqrt 5}2+o(1)}.$$
\end{thm}
Theorem \ref{sqrt5} is weaker than Theorem \ref{sqrt2}, but we have included it as a separated theorem because the simplicity of the construction.
In Theorem \ref{sqrt2} we construct explicitly a denser  infinite Sidon sequence $A=(a_p)_{p\in \Pp*}$ adding the deletion technique to our method. The starting point is  the sequence $A_{\overline q,c}=(a_p)_{p\in \Pp}$ with $c=\sqrt 2-1$. This sequence is not a Sidon sequence but we can delete some elements $a_p$ from the sequence  to destroy the repeated sums that could appear. Thus,  the final set of indexes of our sequence will be  not the whole set of the prime numbers, as in the  construction of Theorem \ref{sqrt5}, but the set $\Pp^*$ formed by the survived primes after we remove a thin subsequence of the primes that we can describe explicitly.
\begin{thm}\label{sqrt2}
The sequence $A=(a_p)_{p\in \Pp^*}$ constructed in subsection \ref{psqrt2} is an infinite Sidon sequence with $$A(x)=x^{\sqrt2-1+o(1)}.$$
\end{thm}

Note that the exponent of the counting function in the explicit construction of Theorem \ref{sqrt2} is the same  that Ruzsa  obtained in his random construction. Furthermore, it can be checked easily that the algorithm used to construct the Sidon sequence in Theorem \ref{sqrt2} is efficient in the sense that only $O(x^{\sqrt 2-1+o(1)})$ elementary operations are needed to list all the elements $a_p\le x$.

Our approach also generalizes to $B_h$ sequences, that is, sequences such that all the sums $a_1+\cdots +a_h,\ a_1\le \cdots \le a_h$ are distinct. To deal with theses cases, however, we need to introduce a probabilistic argument in an unusual way and it becomes the proof of the following theorem not constructive.

\begin{thm}\label{Bh}
For any $h\ge 3$ there exists an infinite $B_h$ sequence $A$ with $$A(x)=x^{\sqrt{(h-1)^2+1}-(h-1)+o(1)}.$$
\end{thm}

The exponent in Theorem \ref{Bh} is greater than $1/(2h-1)$, that given by the greedy algorithm for $B_h$ sequences. It should be mentioned that R. Tesoro and the author \cite{CT} have proved recently Theorem \ref{Bh} in the cases $h=3$ and $h=4$ using a variant of Ruzsa's approach which makes use of the sequence $(\theta(\p))$ of  arguments of the Gaussian primes $\p=|\p|e^{2\pi i\theta(\p)}$ instead of the sequence $(\log p)$ considered by Ruzsa. However, that proof does not extend  to all $h$.

In the last section we present an alternative method to construct infinite Sidon sequences. It has the same flavor than the construction described in section 2 but the irreducible polynomials in $\F_2[X]$ play the role of the prime numbers in the set of positive integers.  The finite version of this construction is the following. We identify $\F_{2^n}\simeq \F_2[X]/q(X)$ where $q:=q(X)$ is an irreducible polynomial  in $\F_2[X]$ with $\deg q=n$. Let $g$ a generator of $\F_{2^n}$ and $g:=g(X)$ the corresponding polynomial
$$A=\{x(p):\ g^{x(p)}\equiv p\pmod{q},\ p \text{ irreducible in }\F_2[X],\  \deg(p)<n/2\}$$
is a Sidon set in $\Z_{2^n-1}$ of size $|A|\gg 2^{n/2}/n$.

We present an sketch of how to reprove Theorems \ref{sqrt5},\ \ref{sqrt2} and \ref{Bh} using this alternative approach.

\section{The construction}Let us consider the following well known fact, which will be used in the construction of our sequence:

 \emph{Given an infinite sequence of positive integers $\overline b:=b_1,\dots ,b_j,\dots$ (the base), any non negative integer $n$ can be written, in an only way, in the form $$n=x_1+x_2b_1+x_3b_1b_2+\cdots +x_jb_1\cdots b_{j-1}+\cdots$$
with digits $0\le x_j<b_{j},\ j\ge 1$ and we represent it as $n:=x_k\dots x_1.$}
\subsection{Construction of the sequence $A_{\overline q,c}$}
We consider the basis $\overline q:= 4q_1,\dots, 4q_j\dots $ where
$q_1,\dots, q_j,\dots $ is a given  infinite  sequence of prime numbers satisfying \begin{equation}\label{q}2^{2j-1}<q_j\le 2^{2j+1}\end{equation} for all $j\ge 1$.
Choose, for each $j$, a primitive root $g_j$ of $\mathbb F_{q_j}^*$.

Fix $c,\ 0<c<1/2$ and consider the partition of the set of the prime numbers, $$\Pp=\bigcup_{k\ge 2} \Pp_k,\quad \text{ where } \quad \Pp_k=\left \{p\text{ prime}:\ 2^{c(k-1)^2-3}<p\le 2^{ck^2-3}\right \}.$$

Let us construct the sequence $A_{\overline q,c}=(a_p)_{p \in \Pp}$ as follows: for $p\in \Pp_k$,  we define
$$a_p=x_k(p)\dots x_1(p),  $$
where the digit $x_j(p)$ in the basis $\overline q$ is the solution of the congruence
\begin{eqnarray}\label{g}g_j^{x_j(p)}\equiv p\pmod{q_j},\qquad q_j+1\le x_j(p)\le 2q_j-1.\end{eqnarray} We define $x_j(p)=0$ when $j>k$.
\subsection{Properties of the sequence $A_{\overline q,c}$}
\begin{prop}\label{growing} All the elements $a_p$ of the sequence $A_{\overline q,c}$ are distinct and the counting function  satisfies $A_{\overline q ,c}(x)=x^{c+o(1)}$.
\end{prop}
\begin{proof}

Suppose that $a_p=a_{p'}$. Thus $x_j(p)=x_j(p')$ for all $j\ge 1$ and we have, by construction, that  $p,p'\in \Pp_k$ where $k$ is the largest $j$ such that $x_j(p)\ne 0$. We also  know that $$p\equiv g^{x_j(p)}\equiv g^{x_j(p')}\equiv p'\pmod{q_j}$$ for all $j\le k$ and then, $p\equiv p' \pmod{q_1\cdots q_{k}}$.  If $p\ne p'$ we would have  $$2^{ck^2}\ge |p-p'|\ge q_1\cdots q_k>2^{1+3+\cdots+(2k-1)}= 2^{k^2},$$ which is impossible because $c<1$.

To study the growing of the sequence $A_{\overline q,c}$, we consider, for any $x$, the integer $k$ such that  $$(4q_1)\cdots (4q_{k})<x\le (4q_1)\cdots (4q_{k+1}).$$ Using \eqref{q} we can check that $2^{k^2+2k}<x\le 2^{(k+2)^2+2(k+1)}$ and then $2^{k^2}=x^{1+o(1)}$.

We observe that if $p\le 2^{ck^2-3}$ then $p\in \Pp_l$ for some $l\le k$ and then $$a_p=x_1(p)+\sum_{j\le l}x_j(p)(4q_1)\cdots (4q_{j-1})\le  (4q_1)\cdots (4q_k)\le x.$$ Thus $$A_{\overline q,c}(x)\ge \pi(2^{ck^2-3})=  2^{ck^2(1+o(1))}=x^{c+o(1)}.$$
For the upper bound we observe that if $p>2^{c(k+1)^2-3}$ then $p\in \Pp_l$ for some $l\ge k+2$ and then $a_p>(4q_1)\cdots (4q_{l-1})\ge(4q_1)\cdots (4q_{k+1})\ge x.$ Thus
$$A_{\overline q,c}(x)\le \pi(2^{c(k+2)^2-3})=2^{ck^2(1+o(1))}=x^{c+o(1)}.$$
\end{proof}

The next proposition concerns to the Sidoness quality of the sequence $A_{\overline q,c}$.
\begin{prop}\label{con} Suppose that there exist $a_{p_1},a_{p_2},a_{p_1'},a_{p_2'}\in A_{\overline q,c},\ a_{p_1}>a_{p_1'}\ge a_{p_2'}>a_{p_2}$ such that $$a_{p_1}+a_{p_2}=a_{p_1'}+a_{p_2'}.$$ Then we have that:
\begin{enumerate}
  \item[i)]  there exist $k_2,k_1,\ k_2\le k_1$ such that $ p_1,p_1'\in \Pp_{k_1},\ p_2,p_2'\in \Pp_{k_2}$.
  \item[ii)] $p_1p_2\equiv p_1'p_2'\pmod{q_1\cdots q_{k_2}}$
  \item[iii)] $\quad p_1\equiv p_1'\pmod{q_{k_2+1}\cdots q_{k_1}} \text{ if } k_2<k_1.$
\item[iv)] $(1-c)k_1^2<k_2^2<\frac c{1-c}k_1^2.$
\end{enumerate}
\end{prop}
\begin{proof}
Since $0\le x_j(p_1)+x_j(p_2)<4q_j$ for all $j$, the equality $a_{p_1}+a_{p_2}=a_{p_1'}+a_{p_2'}$ implies that the digits of both sums are equal:
 \begin{equation}\label{igualdad}x_j(p_1)+x_j(p_2)=x_j(p_1')+x_j(p_2')\end{equation} for all $j$.
 By construction, $p_1\in  \Pp_{k_1}$ and $p_2\in \Pp_{k_2}$ where $k_1$ is the largest $j$ such that $$x_j(p_1)+x_j(p_2)\ge q_j+1$$ and $k_2$ is the largest $j$ such that $$x_j(p_1)+x_j(p_2)\ge 2q_j+2.$$ This observation proves part i) and we can represent $a_{p_1},a_{p_2},a_{p_1'},a_{p_2'}$ as \begin{eqnarray*}a_{p_1}&=&x_{k_1}(p_1)\dots x_{k_2}(p_1)\dots x_1(p_1) \\
a_{p_2}&=&\qquad \qquad   \ x_{k_2}(p_2)\dots x_1(p_2)\\
a_{p_1'}&=&x_{k_1}(p_1')\dots x_{k_2}(p_1')\dots x_1(p_1') \\
a_{p_2'}&=&\qquad \qquad   \ x_{k_2}(p_2')\dots x_1(p_2').
\end{eqnarray*}To prove parts ii) and iii) first we observe that \eqref{igualdad} implies that for all $j$ we have
$$g_j^{x_j(p_1)+x_j(p_2)}\equiv g_j^{x_j(p_1')+x_j(p_2')}\pmod{q_j}.$$ We also know that if $p\in \Pp_k$, then
$g_j^{x_j(p)}\equiv p\pmod{q_j}$ for $j\le k$  and $g_j^{x_j(p)}\equiv 1\pmod{q_j}$ when $j> k.$

Thus, for $j\le k_2$ we have that
$p_1p_2\equiv p_1'p_2'\pmod{q_j}$  and then $$p_1p_2\equiv p_1'p_2'\pmod{q_1\cdots q_{k_2}}.$$

If $k_2<k_1$, for $k_2+1\le j\le k_1$ we have that
$p_1\equiv p_1'\pmod{q_j}$ and then $$p_1\equiv p_1'\pmod{q_{k_2+1}\cdots q_{k_1}}.$$

Part ii) and the inequalities on $p_i$ and $q_j$ yield
$$2^{c(k_1^2+k_2^2)}\ge |p_1p_2-p_1'p_2'|\ge q_1\cdots q_{k_2}>2^{1+3+\cdots +(2k_2-1)}=2^{k_2^2}\Longrightarrow k_2^2<\frac c{1-c}k_1^2.$$ In particular it implies that $k_2<k_1$ and we can apply  part iii), which gives   $$2^{ck_1^2}\ge |p_1-p_1'|\ge q_{k_2+1}\cdots q_{k_1}>2^{(2k_2+1)+\cdots +(2k_1-1)}=2^{k_1^2-k_2^2} \Longrightarrow k_2^2>(1-c)k_1^2.$$
\end{proof}

\subsection{Proof of Theorem \ref{sqrt5} } To prove the first part of Theorem \ref{sqrt5} we simply observe that if there is a repeated sum, then  Proposition \ref{con},  iv) implies that $1-c<\frac c{1-c}$, which does not hold for $c=\frac {3-\sqrt 5}2$. Thus, $A_{\overline q,c}$ is a Sidon sequence for this value of $c$.

\subsection{Proof of Theorem \ref{sqrt2}}\label{psqrt2}
Proposition \ref{con} implies that all the repeated sums that may appear  are of the form:
\begin{equation}\label{ii}a_{p_1}+a_{p_2}=a_{p_1'}+a_{p_2'},\qquad \{p_1,p_2\}\ne \{p_1',p_2'\}\end{equation}
with $p_1,p_1'\in \Pp_{k_1}, p_2,p_2'\in \Pp_{k_2}$ and $ k_2^2<\frac c{1-c}k_1^2$.

Using the notation $Q_1=q_1\cdots q_{{k_2}}$ and $ Q_2=q_{{k_2}+1}\cdots q_{k_1}$ we can write \begin{eqnarray*}p_1(p_2-p_2')&=&p_1p_2-p_1'p_2'+(p_1'-p_1)p_2'\\&=&\frac{p_1p_2-p_1'p_2'}{Q_1}Q_1+\frac{(p_1'-p_1)  }{Q_2}p_2'Q_2.\end{eqnarray*}
Proposition \ref{con} also implies that if \eqref{ii} holds then $s_1=\frac{p_1p_2-p_1'p_2'}{Q_1}$ and $s_2=\frac{p_1'-p_1  }{Q_2}$ are nonzero integers satisfying $$|s_1|=\frac{|p_1p_2-p_1'p_2'|}{Q_1}\le \frac{2^{c(k_1^2+k_2^2)-6}}{Q_1},\qquad |s_2|=\frac{|p_1'-p_1| }{Q_2}\le \frac{2^{ck_1^2-3}}{Q_2}.$$
Thus, if $p_1\in \Pp_{k_1}$ is involved in sum repeated sum as in \eqref{ii} then it divides some integer $s\ne 0$ of the set
$$S_{k_2,k_1}=\left \{s=s_1Q_1+s_2p_2'Q_2:\ 1\le |s_1|\le \frac{2^{c(k_1^2+k_2^2)-6}}{Q_1},\ 1\le |s_2|\le \frac{2^{ck_1^2-3}}{Q_2},\ p_2'\in \Pp_{k_2}\right \}.$$
We denote by $\B_{k_1}$ the set of these primes:
 $$\B_{k_1}=\left \{p_1\in \Pp_{k_1}:\ p_1\mid s \text{ for some } s\in S_{k_2,k_1},\ s\ne 0, \text{ with } k_2^2<\tfrac c{1-c}k_1^2\right \}.$$
 From the analysis above it is clear that for $$\Pp^*=\bigcup_{k_1}\left (\Pp_{k_1}\setminus \B_{k_1}\right )$$ the sequence $A_{\overline q,c}^*=(a_p)_{p\in \Pp^*}$ is a Sidon sequence. To prove that $A_{\overline q,c}^*(x)=x^{c+o(1)}$ it is enough to  show that $|\B_{k_1}|\le(\frac 12+o(1))|\Pp_{k_1}|$. We will prove that this holds for $c=\sqrt 2-1$.

 First we observe that an integer $s\ne 0$ of $S_{k_2,k_1}$ cannot be dived by two primes $p,p'\in \Pp_{k_1}$. Otherwise,
we would have that $$2^{2c(k_1-1)^2-6}<pp'\le |s|\le 2\cdot 2^{c(k_1^2+k_2^2)-6}<2^{\frac c{1-c}k_1^2-5},$$
which does not hold for $k_1$ large enough since $2c>\frac c{1-c}$ for $c<1/2$.

Therefore, using the estimates $Q_1Q_2=q_1\dots q_{k_1}>2^{k_1^2}$ and $|\Pp_{k_2}|=\pi(2^{ck_2^2})\le 2\cdot 2^{ck_2^2}/\log (2^{ck_2^2})$ we have
\begin{eqnarray*}|\B_{k_1}|&\le & \sum_{k_2<\sqrt{\frac c{1-c}}k_1}|S_{k_2,k_1}|\le \sum_{k_2<\sqrt{\frac c{1-c}}k_1}\left (2\cdot \frac{2^{c(k_2^2+k_1^2)-6}}{Q_1}\right )\left (2\cdot \frac{2^{ck_1^2-3}}{ Q_2}\right )|\Pp_{k_2}|\\ &\le &\frac{2^{-6}}{c\log 2}\ 2^{(2c-1)k_1^2}\sum_{k_2<\sqrt{\frac c{1-c}}k_1}\frac{2^{2ck_2^2}}{k_2^2}\le \frac{2^{-6}}{c\log 2}\cdot \frac{2(1-c)}c\cdot  \frac{2^{\left (\frac{2c}{1-c}-1\right)k_1^2}}{k_1^2}.\end{eqnarray*}

Then, using the the identities $\frac{2c}{1-c}-1=c$ and $\frac {1-c}c=\sqrt 2$ for $c=\sqrt 2-1$ and the estimate \begin{eqnarray*}|\Pp_{k_1}|=\pi\left (2^{ck_1^2-3}\right )-\pi\left (2^{c(k_1-1)^2-3}\right )&= &\frac{2^{ck_1^2-3}}{c k_1^2\log 2}(1+o(1))\end{eqnarray*}we have the wanted inequality,
\begin{eqnarray*}|\B_{k_1}|\le  \frac{2^{-5}\sqrt 2}{c\log 2} \frac{2^{ck_1^2}}{k_1^2}\le \left (\frac{1}2+o(1)\right )|\Pp_{k_1}|.\end{eqnarray*}

\section{Infinite $B_h$ sequences} In the following we  shall use the same notation with only minor changes. We consider the basis $\overline q:= h^2q_1,\dots,h^2q_j\dots $ where the primes $q_j$ satisfy $2^{2j-1}<q_j\le 2^{2j+1}$. For each $j$, let $g_j$ be a generator of $\F_{q_j}^*$.

Fix $c=\sqrt{(h-1)^2+1}-(h-1)$ and let $\Pp=\cup_{k\ge 3}\Pp_k$ where
$$\Pp_k=\left \{p \text{ prime }:\ 2^{c(k-1)^2\left (1-1/\sqrt{\log(k-1)}\right )}< p\le 2^{ck^2\left (1-1/\sqrt{\log k}\right )}\right \}.$$ We construct the sequence $A_{\overline q,c}=(a_p)_{p\in \Pp}$ as follows: for $p\in \Pp_k$ we define
$$a_p=x_k(p)\dots x_1(p),$$
where he digit $x_j(p)$ is the solution of the congruence
\begin{eqnarray*}g_j^{x_j(p)}\equiv p\pmod{q_j},\qquad (h-1)q_j+1\le x_j(p)\le hq_j-1.\end{eqnarray*} We define $x_j(p)=0$ for $j>k$.

We observe that the sequence $A_{\overline q,c}=(a_p)_{p\in \Pp}$ will be a $B_h$ sequence if and only if for any $l,\ 2\le l\le h$ there not exists a repeated sum in the form
\begin{eqnarray}\label{rep}
a_{p_1}+\cdots +a_{p_l}&=&a_{p_1'}+\cdots +a_{p_l'}\\
\{a_{p_1},\dots,a_{p_l}\}&\cap& \{a_{p_1'},\dots ,a_{p_l'}\}=\emptyset\nonumber\\
a_{p_1}\ge &\cdots &\ge a_{p_l}\nonumber\\
a_{p_1'}\ge &\cdots & \ge a_{p_l'}.\nonumber
\end{eqnarray}

The following proposition is just a generalization of Proposition \ref{con}.
\begin{prop} \label{con2}Suppose that there exist $p_1,\dots, p_l,p_1',\dots,p_l'\in A_{\overline q,c}$ satisfying \eqref{rep}. Then we have:
\begin{enumerate}
  \item[i)] $p_i,p_i'\in \Pp_{k_i},\ i=1,\dots,l$ for some $k_l\le \cdots \le k_1$.
  \item[ii)]
  $$\begin{matrix}
p_1\cdots p_l&\equiv &p_1'\cdots p_l'&\pmod{q_1\cdots q_{k_l}}&\\
p_1\cdots p_{l-1}&\equiv &p_1'\cdots p_{l-1}'&\pmod{q_{k_l+1}\cdots q_{k_{l-1}}}& \text{ if }k_l<k_{l-1}\\
\cdots & & \cdots\\
p_1&\equiv &p_1'&\pmod{q_{k_2+1}\cdots q_{k_1}}& \text{ if } k_2<k_1.\end{matrix}$$
\item[iii)] $k_l^2< \frac c{1-c}\left (k_1^2+\cdots +k_{l-1}^2\right ).$
\item[iv)] $q_1\cdots q_{k_1}\mid \prod_{i=1}^l\left (p_1\cdots p_i-p_1'\cdots p_i'\right ).$
\end{enumerate}
\end{prop}
\begin{proof}
The proof is similar to the proof of Proposition \ref{con}. Here $k_i$ is the largest $j$ such that $x_j(p_1)+\cdots +x_j(p_l)\ge i((h-1)q_j+1)$. Part iii) is consequence of the first congruence of part ii). Part iv) is also an obvious consequence of part ii).
\end{proof}

\subsection{Proof of Theorem \ref{Bh}}
The sequence $A_{\overline q,c}$ defined at the beginning of this section may  not be a $B_h$ sequence. The plan of the proof is to remove from $A_{\overline q,c}=(a_p)_{p\in \Pp}$ the largest element appearing in each such repeated sum to obtain a true $B_h$ sequence.

More precisely, we define $\Pp^*=\Pp^*(\overline q)$ as the set
$$\Pp^*=\bigcup_k\left (\Pp_k\setminus \B_k(\overline q)\right )$$
where
$\B_k(\overline q)=\{p\in \Pp_k:\  a_p \text{ is the largest involved in some equation }\eqref{rep}\}.$

It is then clear that the sequence
$A^*_{\overline q,c}=(a_p)_{p\in \Pp^*}$
is a $B_h$ sequence.

We can proceed as in the previous section to deduce that
$A_{\overline q,c}(x)=x^{c+o(1)}$. If in addition, $|\B_k(\overline q)|=o(|\Pp_k|)$, we have that $$A^*_{\overline q,c}(x)\sim A_{\overline q,c}(x)=x^{c+o(1)}.$$
Thus, the proof of Theorem \ref{Bh}  will be completed if we prove that there exists a basis $\overline q$ such that $|\B_k(\overline q)|=o(|\Pp_k|)$.

For $2\le l\le h$ we write $$\text{Bad}_l(\overline q,k_l,\dots,k_1)=\{(p_1,\dots ,p_l'):\ p_i,p_i'\in \Pp_{k_i},\ i=1,\dots,l\ \text{satisfying } \eqref{rep}\}.$$
Next let us observe that each  $p\in \B_k(\overline q)$ comes from some $(p_1,\dots, p_l')\in \text{Bad}_l(\overline q,k_l,\dots,k_1)$, $ 2\le l\le h, \ k_l\le \cdots \le k_1=k$. Thus,
\begin{eqnarray}\label{bad}|\B_k(\overline q)|&\le& \sum_{l=2}^h\sum_{k_l\le \cdots \le k_1=k}|\text{Bad}_l(\overline q,k_l,\dots,k_1)|\\
&\le &hk^{h-1}\max_{\substack{2\le l\le h\\ k_l\le \cdots \le k_1=k}}|\text{Bad}_l(\overline q,k_l,\dots,k_1)|.\nonumber\end{eqnarray}

It happens that we are not able to give a good upper bound for $|\text{Bad}_l(\overline q,k_l,\dots,k_1)|$ for a concrete sequence $\overline q:=q_1<q_2<\cdots $, but we can do it in average. If the reader is familiarized with Ruzsa's work, the basis $\overline q$ will play the same role as the real parameter $\alpha$ in Ruzsa's construction.

We consider the probability space of the basis  $\overline q:=h^2q_1,h^2q_2\dots  $ where each $q_j$ is chosen at random uniformly between all the primes in the interval $(2^{2j-1},2^{2^j+1}]$. In particular we  use that $\pi(2^{2k+1})-\pi(2^{2k-1})\gg 2^{2k}/k=2^{2k-1+O(\log k)}$ to deduce that  for any primes $q_1<\cdots <q_{k_1}$ satisfying $2^{2j-1}<q_j\le 2^{2j+1}$ we have
\begin{eqnarray*}\mathbb P( h^2q_1,\dots,h^2q_{k_1}\in \overline q)&=&\prod_{k=1}^{k_1}\frac 1{\pi(2^{2k+1})-\pi(2^{2k-1})}\\
&\le &2^{-k_1^2+O(k_1\log k_1)}.\end{eqnarray*}
Thus, for a given $(p_1,\dots,p_l')$, we use Proposition \ref{con2}, iv) and the estimate $\tau(n)=n^{O(1/\log\log n)}$ for the divisor function to deduce  that
\begin{eqnarray*}\mathbb P( (p_1,\dots,p_l')\in \text{Bad}_l(\overline q,k_l,\dots,k_1))&\le &\sum_{\substack{q_1,\dots,q_{k_1}\\q_1\cdots q_{k_1}\mid \prod_{i=1}^l\left (p_1\cdots p_i-p_1'\cdots p_i'\right )}}
 \mathbb P( q_1,\dots,q_{k_1}\in \overline q)\\ &\le &\tau \left (\prod_{i=1}^l\left (p_1\cdots p_i-p_1'\cdots p_i'\right )\right )2^{-k_1^2+O(k_1\log k_1)}\\
&\le &2^{-k_1^2+O(k_1^2/\log k_1)}.\end{eqnarray*}

Thus, using Proposition \ref{con2} iii) in the last inequality we have:
\begin{eqnarray*}\E(|\{(p_1,\dots ,p_l')& : & \ p_i,p_i'\in \Pp_{k_i},\ i=1,\dots ,l \text{ satisfying }\eqref{rep}\}|)\\
&\le &2^{-k_1^2+O(k_1^2/\log k_1)}\#\{(p_1,\dots,p_l'):\ p_i,p_i'\in \Pp_{k_i}\}\\
&\le &2^{-k_1^2+O(k_1^2/\log k_1)}|\Pp_{k_1}|^2\cdots |\Pp_{k_l}|^2\\
&\le &2^{-k_1^2+O(k_1^2/\log k_1)}\cdot 2^{{2c}\left (k_1^2+\cdots +k_l^2\right )-2ck_1^2/\sqrt{\log k_1}}\\
&\le& 2^{-k_1^2+\frac {2c}{1-c}\left (k_1^2+\cdots +k_{l-1}^2\right )-(2c+o(1))k_1^2/\sqrt{\log k_1}}\\
&\le& 2^{\left (-1+\frac{2c(l-1)}{1-c}\right )k_1^2-(2c+o(1))k_1^2/\sqrt{\log k_1}}.\end{eqnarray*}

Using \eqref{bad} we   have
\begin{eqnarray*}\E(|\B_k(\overline q)|)\le 2^{\left (-1+\frac{2c(h-1)}{1-c}\right )k^2-(2c+o(1))k^2/\sqrt{\log k}}.\end{eqnarray*}

Finally we use that $-1+\frac{2c(h-1)}{1-c}-c=0$ for $c=\sqrt{(h-1)^2+1}-(h-1)$ to  obtain
\begin{eqnarray*}\E\left (\sum_k\frac{|\B_k(\overline q)|}{|\Pp_k|}\right )&\le &\sum_k k^22^{\left (-1+\frac{2c(h-1)}{1-c}-c\right )k^2-(c+o(1))k^2/\sqrt{\log k}}\\ &\le &\sum_k k^22^{-(c+o(1))k^2/\sqrt{\log k}} .\end{eqnarray*}
Since the series is convergent we have that for almost all sequences $\overline q$ the series $$\sum_k\frac{|\B_k(\overline q)|}{|\Pp_k|}$$ is convergent. Therefore, for any of these  $\overline q$ we have that $|\B_k(\overline q)|=o(|\Pp_K|),$ which is what we wanted to prove.

\section{An alternative construction} The theorems proved in this paper could have been proved using the following alternative construction, which, although  has the same flavor than the one described in section 2, it uses the irreducible polynomials in $\F_2[X]$ instead of the  prime numbers.

We consider the basis $\overline q:=4\cdot 1,4\cdot 3,\dots, 4\cdot (2j-1),\dots $ and any infinite  sequence of irreducibles polynomials in $\F_2[X]$ of degree $\deg(q_j)=2j-1$. For each $j$, let $g_j(X)$ a generator of  $\F_2[X]/q_j(X)$. Fix $c,\ 0<c<1/2$ and for each $k\ge 2$, let  \begin{equation}\label{PK}\Pp_k=\left \{p\text{ irreducible polynomials in }\F_2[X]:\ c(k-1)^2<\deg p\le ck^2\right \}.\end{equation}

Consider the sequence $A_{\overline{q},c}=(a_p)$ where,
for any $p(X)\in \Pp_k$ (we write $p:=p(X)$ for short), the element $a_p$ is defined by
$$a_p:=x_k(p)\dots x_1(p)$$ where
 $x_j(p)$  is the solution of the polynomial congruence
$$g_j(X)^{x_j(p)}\equiv p(X)\pmod{q_j(X)},\qquad 2^{2j-1}+1\le x_j(p)\le 2^{2j}-1.$$ Let us define $x_j(p)=0$ for $j> k$.
More formally we can write
 \begin{equation}\label{bp}a_p=\sum_{1\le j\le k}x_j(p)2^{j^2-1}.\end{equation}
We use that the number of irreducible polynomials of degree $j$ in $\F_2[X]$ is $\gg 2^j/j$ to deduce easily that in this case we also have  $A_{\overline q,c}(x)=x^{c+o(1)}$. Proposition \ref{con}  also works here, except that now the congruences are in $\F_2[X]$.
It is then  easy to adapt the proofs of Theorems \ref{sqrt5} and  \ref{sqrt2} to this new construction.

The proof of Theorem \ref{Bh} using this construction is also similar to that given in section 2, except that now we consider the basis $\overline q=h^2\cdot 1,h^2\cdot 3,\dots,h^2\cdot(2j-1),\dots$ and define $$a_p=x_k(p)\dots x_1(p)$$
where $x_j(p)$  is the solution of the  congruence
$$g_j^{x_j(p)}\equiv p(X)\pmod{q_j(X)},\qquad (h-1)2^{2j-1}+1\le x_j(p)\le h2^{2j-1}-1.$$

Perhaps,  the less known ingredient needed in the proof may be the upper bound $\tau(r(X))\le 2^{O(n/\log n)}$ for the number of the  divisors of a polynomial $r(X)\in \F_2[X]$ of degree $n$ (see \cite{P}).

\end{document}